\documentclass[a4paper,10pt,twoside]{article}
\pdfoutput=1
\usepackage[T1]{fontenc} %WARNING make sure you have the package "cm-super" installed, or this line will produce ugly bitmap fonts. If you cannot install cm-super, probably it is better to comment this line
\usepackage[utf8]{inputenc}
\usepackage[british]{babel}
\usepackage{amsfonts}
\usepackage{amsmath} 
\usepackage{amsthm}
\usepackage{amssymb}
\usepackage{mathtools}
\usepackage{tikz}
\usetikzlibrary{matrix,arrows,decorations.pathmorphing, decorations.markings}

\usepackage[affil-it]{authblk}
\usepackage{hyperref}
\hypersetup{colorlinks=false, pdfborder={0 0 0}}
\usepackage{fancyhdr}                                   

\usepackage{url}
\usepackage{cite}
\usepackage{enumitem}
\usepackage{float}
\makeatletter                        
\g@addto@macro{\UrlBreaks}{\UrlOrds} 
\makeatother

\newcommand{\from}{\colon}
\renewcommand{\phi}{\varphi}
\renewcommand{\models}{\vDash}
\newcommand{\proves}{\vdash}
\newcommand{\restr}{\upharpoonright}
\newcommand{\implica}{\rightarrow}
\newcommand{\coimplica}{\leftrightarrow}
\newcommand{\bla}[4]{{#1}_{#2}#3\ldots#3{#1}_{#4}}
\DeclareMathOperator{\Th}{Th}
\DeclareMathOperator{\Con}{Con}
\DeclareMathOperator{\dom}{dom}

\DeclarePairedDelimiter{\set}{\{}{\}}

\theoremstyle{definition}
\newtheorem{defin}{Definition}[section]
\newtheorem{thm}[defin]{Theorem}
\newtheorem*{unnmdthm}{Theorem}
\newtheorem{pr}[defin]{Proposition}
\newtheorem{co}[defin]{Corollary}
\newtheorem{lemma}[defin]{Lemma}
\newtheorem*{notation}{Notation}
\newtheorem{rem}[defin]{Remark}
\newtheorem{fact}[defin]{Fact}

\newtheorem{assumption}[defin]{Assumption}
\newtheorem{eg}[defin]{Example}

\makeatletter
\newcommand{\subjclass}[2][2010]{%
  \let\@oldtitle\@title%
  \gdef\@title{\@oldtitle\footnotetext{\noindent #1 \emph{Mathematics subject classification.} #2}}%
}
\newcommand{\keywords}[1]{%
  \let\@@oldtitle\@title%
  \gdef\@title{\@@oldtitle\footnotetext{\noindent \emph{Key words and phrases.} #1.}}%
}
\makeatother

\author[]{Bea Adam-Day%
  \thanks{email: \url{B.Adam-Day@leeds.ac.uk} \textsc{orcid}: \url{https://orcid.org/0000-0002-7891-916X}} }
\author[]{John Howe%
  \thanks{email: \url{J.A.Howe@leeds.ac.uk} \textsc{orcid}: \url{https://orcid.org/0000-0001-7580-6349}
} }
\author[]{Rosario Mennuni%
  \thanks{email: \url{R.Mennuni@leeds.ac.uk} \textsc{orcid}: \url{https://orcid.org/0000-0003-2282-680X}}}
\affil[]{University of Leeds}

\title{On Double-Membership Graphs \\of Models of Anti-Foundation}
\begin{document}
\maketitle
\begin{abstract}\noindent We answer some questions about graphs which are reducts of countable models of Anti-Foundation, obtained by considering the binary relation of double-membership $x\in y\in x$.  We show that there are continuum-many such graphs, and study their connected components. We describe their complete theories and prove that each  has continuum-many countable models, some of which are not reducts of models of Anti-Foundation.
\end{abstract}

\noindent  This paper is concerned with the model-theoretic study of a class of graphs arising as reducts of a certain non-well-founded set theory.

Ultimately, models of a set theory are digraphs, where a directed edge between two points denotes membership. To such a model, one can associate various graphs, such as the \emph{membership graph}, obtained by symmetrising the binary relation $\in$, or the \emph{double-membership graph}, which has an edge between $x$ and $y$ when  $x\in y$ and $y\in x$ hold simultaneously. We also consider the structure equipped with the two previous graph relations, which we call the \emph{single-double-membership graph}.  In~\cite{adamdaycameron} the first author and Peter Cameron investigated this kind of object in the non-well-founded case. We continue this line of study, and answer some questions regarding such graphs which were  left open in the aforementioned work.

  It is well-known that every membership graph of a countable model of $\mathsf{ZFC}$ is isomorphic to the Random Graph (see e.g.~\cite{cameron}). The usual proof of this fact goes through for set theories much weaker than $\mathsf{ZFC}$, but uses the Axiom of Foundation in a crucial way, hence the interest in (double-)membership graphs of non-well-founded set theories. 

  Perhaps the most famous of these is $\mathsf{ZFA}$\footnote{Note that, in the literature, the name $\mathsf{ZFA}$ is also used for a certain variant of $\mathsf{ZFC}$ which admits urelements.}, obtained from $\mathsf{ZFC}$ by replacing the Axiom of  Foundation with the \emph{Anti-Foundation Axiom}. This axiom was explored in, amongst others,~\cite{fortihonsell, aczel, barwisemoss}. It provides a rich class of non-well-founded sets, the structure of which reflects that of the well-founded sets.
 In every model of Anti-Foundation there will be, for example, unique sets $a$ and $b$ such that $a=\set{b, \emptyset}$ and  $b=\set{a,\set{\emptyset}}$, and a unique $c=\set{c, \emptyset, \set{\emptyset}}$. These are pictured in Figure~\ref{figure:afaexample}.
\begin{figure}[b]
  \centering \caption{On the left, a picture of the unique sets $a$ and $b$ such that  $a=\set{b, \emptyset}$ and  $b=\set{a,\set{\emptyset}}$. On the right, a picture of the unique set $c$ such that $c=\set{c,\emptyset, \set{\emptyset}}$. The arrows denote membership.}\label{figure:afaexample}
\begin{tikzpicture}[scale=2]
\node(a) at (0,0){$a$};
\node(b) at (1,0){$b$};
\node(em) at (0,-0.75){$\emptyset$};
\node(one) at (1,-0.75){$\set\emptyset$};

\node(c) at (3,0){$c$};
\node(em1) at (3,-0.75){$\emptyset$};
\node(one1) at (4,-0.75){$\set\emptyset$};

\path[draw,decoration={
    markings,
    mark=at position 0.5 with {\arrow[black]{stealth};}}]
(a) edge[bend left, postaction=decorate]  (b)
(b) edge[bend left, postaction=decorate]  (a)
(em) edge[postaction=decorate] (a)
(one) edge[postaction=decorate] (b)
(em) edge[postaction=decorate] (one)
;
\path[draw,decoration={
    markings,
    mark=at position 0.5 with {\arrow[black]{stealth};}}]
(em1) edge[postaction=decorate] (c)
(one1) edge[postaction=decorate] (c)
(em1) edge[postaction=decorate] (one1)
;

\draw [decoration={
    markings,
    mark=at position 0.55 with {\arrow[black, rotate=-10]{stealth};}}, postaction=decorate] (c) ++ (-0.01, 0.08) arc [start angle=30, end angle=310, radius=0.13];
\end{tikzpicture}
\end{figure}
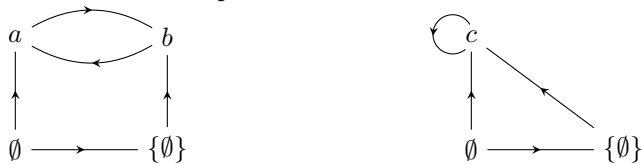

 In~\cite{adamdaycameron} it was proven that all membership graphs of countable models of $\mathsf{ZFA}$ are isomorphic to the `Random Loopy Graph': the Fra\"iss\'e limit of finite graphs with self-edges. This structure is easily seen to be $\aleph_0$-categorical, ultrahomogeneous, and supersimple of SU-rank $1$. On the other hand, double-membership graphs of models of $\mathsf{ZFA}$ are, in a number of senses, much more complicated. For instance,  \cite[Theorem~3]{adamdaycameron}  shows that they are not $\aleph_0$-categorical, and we show further results in this direction.

The structure of the paper is as follows. After setting up the context in Section~\ref{sec:setup}, we answer~\cite[Question~3]{adamdaycameron} in Section~\ref{sec:connectedcomponents} by characterising the connected components of double-membership graphs of models of $\mathsf{ZFA}$. In the same section, we show that if we do not assume Anti-Foundation, but merely drop Foundation, then double-membership graphs can be almost arbitrary. Section~\ref{sec:cmany} answers~\cite[Questions~1 and~2]{adamdaycameron} by proving the following theorem.
\begin{unnmdthm}[Corollary~\ref{co:manymodels}]
There are, up to isomorphism,   continuum-many countable (single-)double-membership graphs of models of $\mathsf{ZFA}$, and continuum-many countable models of each of their theories. 
\end{unnmdthm}
\noindent In Section~\ref{sec:theory} we study the common theory of double-membership graphs, which we show to be incomplete.  Then, by using methods more commonly encountered in finite model theory, we characterise the completions of said theory in terms of consistent collections of consistency statements.
\begin{unnmdthm}[Theorem~\ref{thm:completions}]
  The double-membership graphs of two models $M$ and $N$ of $\mathsf{ZFA}$ are elementarily equivalent precisely when $M$ and $N$ satisfy the same consistency statements.
\end{unnmdthm}
\noindent We also show that all of these completions are wild in the sense of neostability theory, since each of their models interprets (with parameters) arbitrarily large finite fragments of $\mathsf{ZFC}$. Our final result, below~---~obtained with similar techniques~---~answers~\cite[Question~5]{adamdaycameron} negatively. The analogous statement for double-membership graphs holds as well.
\begin{unnmdthm}[Corollary~\ref{co:q5}]
  For  every single-double-membership graph of a model of $\mathsf{ZFA}$, there is a countable elementarily equivalent structure which is not the single-double-membership graph of any model of $\mathsf{ZFA}$.
\end{unnmdthm}

\section{Set-Up}\label{sec:setup}

Since Anti-Foundation allows for sets that are members of themselves, in what follows we will need to deal with graphs where there might be an edge between a point and itself. These are called \emph{loopy graphs} in~\cite{adamdaycameron} but,  for the sake of concision, we depart from common usage by adopting the following convention.
\begin{notation}
By \emph{graph} we mean a first-order structure with a single relation which is binary and symmetric (it is not required to be irreflexive).
\end{notation}

There are a number of equivalent formulations of Anti-Foundation. The form which we shall be using is known in the literature (e.g.~\cite[p.~71]{barwisemoss}) as the \emph{Solution Lemma}.    Other formulations are in terms of \emph{homomorphism onto  transitive structures} (axiom $X_1$ from~\cite{fortihonsell}), or \emph{decorations}.   For the equivalence between them, see e.g.~\cite[p.~16]{aczel}.

\begin{defin}\label{defin:flatsystem} Let $X$ be a set of `indeterminates', and $A$ a set of sets. A \emph{flat system of equations} is a set of equations of the form $x = S_x$, where  $S_x$ is a subset of $X \cup A$ for each $x \in X$.   A \emph{solution} $f$ to the flat system is a function taking elements of $X$ to sets, such that after replacing  each $x\in X$ with $f(x)$ inside the system, all of its equations become true.

  The \emph{Anti-Foundation Axiom} ($\mathsf{AFA}$) is the statement that every flat system of equations has a unique solution.
\end{defin}
\begin{eg}
Consider the flat system with $X=\set{x,y}$,  $A=\set{\emptyset,\set{\emptyset}}$ and the following equations.
\[
  \begin{split}	x&=\{y,\emptyset\}\\
	y&=\{x,\set\emptyset\}
      \end{split}
\]
The image of its unique solution $x\mapsto a, y\mapsto b$  is pictured in Figure~\ref{figure:afaexample}.
\end{eg}

Note that solutions of systems need not be injective, and in fact uniqueness sometimes prevents injectivity. For instance, if $x\mapsto a$ is the solution of the flat system consisting of the single equation $x=\set{x}$, then $x\mapsto a, y\mapsto a$ solves the system with equations $x=\set y$ and $y=\set x$, whose unique solution is therefore not injective.

\begin{rem}
There exists a weak form of $\mathsf{AFA}$ that only postulates the existence of solutions to flat systems, but not necessarily their uniqueness, known as axiom $X$ in~\cite{fortihonsell} or $\mathsf{AFA}_1$ in~\cite{aczel}. In what follows and in~\cite{adamdaycameron}  uniqueness is  never used, hence all the results go through for models of $\mathsf{ZFC}$ with Foundation replaced by $\mathsf{AFA}_1$. For brevity, we still state everything for $\mathsf{ZFA}$.
\end{rem}

    \begin{fact}\label{fact:equicon}
$\mathsf{ZFC}$ without the Axiom of Foundation proves the equiconsistency of $\mathsf{ZFC}$ and $\mathsf{ZFA}$.
\end{fact}
\begin{proof}
In one direction, from a model of $\mathsf{ZFA}$ one obtains one of $\mathsf{ZFC}$ by restricting to the well-founded sets. In the other direction, see~\cite[Theorem~4.2]{fortihonsell} for a class theory version, or~\cite[Chapter~3]{aczel} for the $\mathsf{ZFC}$ statement.
\end{proof}

Since we are interested in studying (reducts of) models of $\mathsf{ZFA}$, we need to assume they exist in the first place, since otherwise the answers to the questions we are studying are trivial. Therefore, in this paper we work in a set theory which is slightly stronger than usual.
\begin{assumption}
The ambient metatheory is $\mathsf{ZFC}+\Con(\mathsf{ZFC})$.
\end{assumption}

\begin{defin}\label{defin:symmetrisation}
  Let $L=\set{\in}$, where $\in$ is a binary relation symbol, and $M$ an $L$-structure. Let $S$ and $D$ be the definable relations
  \begin{gather*}
    S(x,y)\coloneqq x\in y\lor y\in x\\
    D (x,y)\coloneqq x\in y\land y\in x
\end{gather*}
The \emph{single-double-membership graph}, or \emph{SD-graph}, $M_0$ of $M$ is the reduct of $M$ to $L_0\coloneqq\set{S,D}$.  The \emph{double-membership graph}, or \emph{D-graph}, $M_1$ of $M$ is the reduct of $M$ to $L_1\coloneqq\set{D}$. 
\end{defin}
So, given an $L$-structure $M$, i.e.\ a digraph (possibly with loops) where the edge relation is $\in$, we have that $M_0\models S(x,y)$ if and only if in $M$ there is at least one $\in$-edge between $x$ and $y$. Similarly  $M_0\models D(x,y)$ means that in $M$ we have both $\in$-edges between $x$ and $y$. The idea is that, if $M$ is a model of some set theory, then $M_0$ is a symmetrisation of $M$ which keeps track of double-membership as well as single-membership, and $M_1$ only keeps track of double-membership.

In~\cite{adamdaycameron},  $M_0$ is called the \emph{membership graph (keeping double-edges)} of $M$ and $M_1$ is called the \emph{double-edge graph} of $M$.  Note that, strictly speaking, SD-graphs are not graphs, according to our terminology.

For the majority of the paper we are concerned with D-graphs, since most of the results we obtain for them imply the analogous versions for SD-graphs. This situation will reverse in Theorem~\ref{thm:noinfdiam}.

\begin{defin}
  Let $M\models \mathsf{ZFA}$. We say that $A\subseteq M$ is an \emph{$M$-set} iff there is $a\in M$ such that $A=\set{b\in M\mid M\models b\in a}$.
\end{defin}

So an $M$-set $A$ is a definable subset of $M$ which is the extension of a set in the sense of $M$, namely the $a\in M$ in the definition. We will occasionally abuse notation and refer to an $M$-set $A$ when we actually mean the corresponding $a\in M$.

\section{Connected Components}\label{sec:connectedcomponents}

Let $M\models \mathsf{ZFA}$. It was proven in~\cite[Theorem~4]{adamdaycameron} that, for every finite connected graph $G$, the D-graph $M_1$ has infinitely many connected components isomorphic to $G$. It was asked in~\cite[Question~3]{adamdaycameron} if more can be said about the infinite connected components of $M_1$. In this section we characterise them in terms of the graphs inside $M$.

Let $G$ be a graph in the sense of $M\models \mathsf{ZFA}$, i.e.\ a graph whose domain and edge relation are $M$-sets, the latter as, say, a set of Kuratowski pairs. If $G$ is such a graph and $M\models \text{`$G$ is connected'}$, then $G$ need not necessarily be connected. This is due to the fact that $M$ may have non-standard natural numbers, hence relations may have non-standard transitive closures.  We therefore introduce the following notion. 

\begin{defin}\label{defin:region}
  Let $a\in M\models \mathsf{ZFA}$. Let $b\in M$ be such that
  \[
    M\models\text{`$b$ is the transitive closure of $\set{a}$ under $D$'}
  \]
      The \emph{region of $a$ in $M$} is $\set{c\in M\mid M\models c\in b}$.  If $A\subseteq M$, we say that $A$ is a \emph{region of $M$} iff it is the region of some $a\in M$.
    \end{defin}

\begin{rem}\label{rem:compdef}
  For each $a\in M$, the region of $a$ in $M$ is an $M$-set.
\end{rem}

For $a\in M$, if $A$ is  the region of $a$ and  $B$ is the transitive closure of $\set{a}$ under $D$ computed in the metatheory, i.e.\ the connected component of $a$ in $M_1$, then $B\subseteq A$. In particular, regions of $M$ are unions of connected components of $M_1$. If $M$ contains non-standard natural numbers and the diameter of $B$ is infinite then the inclusion $B\subseteq A$ may be strict, and $B$ may not even be an $M$-set.  From now on, the words `connected component' will only be used in the sense of the metatheory.

Most of the appeals to $\mathsf{AFA}$ in the rest of the paper will be applications of the following proposition. In fact, after proving it, we will only deal directly with flat systems twice more.

\begin{pr}\label{pr:embedgraphs}
Let $M_1$ be the D-graph of $M\models \mathsf{ZFA}$, and let $G$ be a graph in $M$.   Then there is $H\subseteq M_1$ such that
  \begin{enumerate}
  \item $(H, D^{M_1}\restr H)$ is isomorphic to $G$,
  \item $H$ is a union of regions of $M$, and
  \item $H$ is an $M$-set.
  \end{enumerate}
\end{pr}
\begin{proof}
Work in $M$ until further notice.  Let $G$ be a graph in $M$, say in the language $\set{R}$. Let $\kappa$ be its cardinality, and assume up to a suitable isomorphism that $\dom G=\kappa$. In particular, note that every element of $\dom G$ is a well-founded set. Consider the flat system
  \[
\set{x_i=\set{i, x_j\mid j\in\kappa, G\models R(i,j)}\mid i\in \kappa}
\]
Let $s\from x_i\mapsto a_i$ be a solution to the system. If $i\ne j$, then $i\in a_i\setminus a_j$, and therefore $s$ is injective. Observe that
\begin{enumerate}[label=(\roman*)]
\item \label{point:iso} since $R$ is symmetric, we have $a_i\in a_j\in a_i\iff G\models R(i,j)$, and
\item \label{point:ucc} for all $b\in M$ and all $i\in \kappa$, we have $b\in a_i\in b$ if and only if there is $j<\kappa$ such that $b=a_j$ and $G\models R(i,j)$.
\end{enumerate}

Now work in the ambient metatheory.  Consider the $M$-set
\[
  H\coloneqq\set{a_i\mid M\models i\in \kappa}=\set{b\in M\mid M\models b\in \operatorname{Im}(s)}\subseteq M_1
\]
By~\ref{point:iso} above, $(H, D^{M_1}\restr H)$ is isomorphic to $G$ and, by~\ref{point:ucc} above, $H$ is a union of regions of $M$. 
\end{proof}

We can now generalise~\cite[Theorem~4]{adamdaycameron}, answering~\cite[Question~3]{adamdaycameron}. The words `up to isomorphism' are to be interpreted in the sense of the metatheory, i.e.\ the isomorphism need not be in $M$.
\begin{thm}\label{thm:aq3}
Let $M\models \mathsf{ZFA}$.  Up to isomorphism, the connected components  of $M_1$ are exactly the connected components  (in the sense of the metatheory) of graphs in the sense of $M$. In particular, there are infinitely many copies of each of them.
\end{thm}
\begin{proof}
  Let $C$  be a connected component of a graph $G$ in $M$. By Proposition~\ref{pr:embedgraphs}  there is an isomorphic  copy $H$ of $G$ which is a union of regions of $M$, hence, in particular, of connected components of $M_1$. Clearly, one of the connected components of $H$ is isomorphic to $C$.
  
  In the other direction, let $a\in M_1$ and consider its connected component. Inside $M$, let $G$ be the region of $a$. Using Remark~\ref{rem:compdef} it is easy to see that $(G, D\restr G)$ is a graph in $M$, and one of its connected components is isomorphic to the connected component of $a$ in $M_1$.

  For the last part of the conclusion take, inside $M$, disjoint unions of copies of a given graph.
\end{proof}

If one does not assume some form of $\mathsf{AFA}$ and for instance merely drops Foundation, then double-membership graphs can be essentially arbitrary, as the following proposition shows. 
\begin{pr}
Let $M\models \mathsf{ZFC}$ and let $G$ be a graph in $M$.  There is a model $N$ of $\mathsf{ZFC}$ without  Foundation such that $N_1$ is isomorphic to the union of $G$ with infinitely many isolated vertices, i.e.\ points without any edges or self-loops.
\end{pr}
Note that the isolated vertices are necessary, as $N$ will always contain well-founded sets.
\begin{proof}
  Let $G$ be a graph in $M$, say in the language $\set{R}$. Assume without loss of generality that $G$ has no isolated vertices, and that $\dom G$ equals its cardinality $\kappa$. For each $i\in \kappa$ choose  $a_i\subseteq \kappa$ which has foundational rank $\kappa$ in $M$, e.g.\ let $a_i\coloneqq\kappa\setminus \set i$. Let $b_j\coloneqq\set{a_i\mid G\models R(i,j)}$ and note that, since no vertex of $G$ is isolated,  $b_j$ is non-empty, thus has rank $\kappa+1$.  Define $\pi\from M\to M$ to be the permutation swapping each $a_i$ with the corresponding $b_i$ and fixing the rest of $M$. Let $N$ be the structure with the same domain as $M$, but with membership relation defined as
  \[
    N\models x\in y\iff M\models x\in \pi(y)
  \]
  By~\cite[Section~3]{rieger}\footnote{Strictly speaking,~\cite{rieger} works in class theory. The exact statement we use is that of~\cite[Chapter~IV, Exercise~18]{kunen}.}, $N$ is a model of $\mathsf{ZFC}$ without Foundation. To check that $N_1$ is as required, first observe that
  \[
    N\models a_i\in a_j\iff M\models a_i\in \pi(a_j)=b_j\iff G\models R(i,j)
  \]
  so $\set{a_i\mid M\models i\in \kappa}$, equipped with the restriction of $D^{N_1}$, is isomorphic to $G$. To show that there are no other $D$-edges in $N_1$, assume that $N_1\models D(x,y)$, and consider the following three cases (which are exhaustive since $D$ is symmetric).
  \begin{enumerate}[label=(\roman*)]
  \item $x$ and $y$ are both fixed points of $\pi$. This contradicts Foundation in $M$.
  \item $y=a_i$ for some $i$, so $N\models x\in a_i$, hence $M\models x\in \pi(a_i)=b_i$. Then $x=a_j$ for some $j$ by construction.
  \item $y=b_i$ for some $i$. From $N\models x\in b_i$ we get  $M\models x\in a_i\subseteq \kappa$, thus $x$ has rank strictly less than $\kappa$. Therefore, $x$ is not equal to any $a_j$ or $b_j$, hence $\pi(x)=x$. Again by rank considerations, it follows that $M\models b_i\notin x=\pi(x)$, so $N\models b_i\notin x$, a contradiction.\qedhere
  \end{enumerate}
\end{proof}

\section{Continuum-Many Countable Models}\label{sec:cmany}

We now turn our attention to answering~\cite[Questions 1 and 2]{adamdaycameron}. Namely, we compute, via a type-counting argument, the number of non-isomorphic D-graphs of countable models of $\mathsf{ZFA}$ and the number of countable models of their complete theories. The analogous results for SD-graphs also hold.
\begin{defin}\label{defin:flowers} Let $n\in \omega\setminus\set{0}$. Define the $L_1$-formula
  \begin{multline*}
    \phi_n(x)\coloneqq \neg D(x,x)\land \exists \bla z0,{n-1} \Bigl(\bigl(\bigwedge_{0\le i<j<n}z_i\ne z_j\bigr)\\
   \land \bigl(\bigwedge_{0\le i<n}D(z_i, x)\bigr)\land \bigl(\forall z\; D(z,x)\implica \bigvee_{0\le i< n} z=z_i\bigr)
    \Bigr)
  \end{multline*}
For $A$ a subset of $\omega\setminus\set{0}$, define the set of $L_1$-formulas
  \begin{multline*}
    \beta_{A}(y)\coloneqq \set{\neg D(y,y)}\cup \set{\exists x_n\; \phi_n(x_n)\land D(y,x_n)\mid n\in A}\\\cup  \set{\neg(\exists x_n\; \phi_n(x_n)\land D(y,x_n))\mid n\in \omega\setminus(\set{0}\cup A)}
  \end{multline*}
We say that $a\in M_1$ is an \emph{$n$-flower} iff  $M_1\models \phi_n(a)$. We say that $b\in M_1$ is an \emph{$A$-bouquet} iff for all $\psi(y)\in \beta_{A}(y)$ we have $M_1\models \psi(b)$.
\end{defin}
So $a$ is an $n$-flower if and only if, in the D-graph, it is a point of degree $n$ without a self-loop, while $b$ is an $A$-bouquet iff it has no self-loop, it has $D$-edges to at least one $n$-flower for every $n\in A$, and it has no $D$-edges to any $n$-flower if $n\notin A$.

\begin{figure}[h]
  \centering \caption{The set $a=\set{\set{a,i}\mid i<5}$ is a $5$-flower. The reason for the name `$n$-flower' can be seen in this figure.}\label{figure:5flower}

  \begin{tikzpicture}[scale=2]
   
\node(y) at (0,0){$a$};
\node(0) at (0.59,0.81){$\set{a,0}$};
\node(1) at (-0.59,0.81){$\set{a,1}$};
\node(2) at (-0.95,-0.309){$\set{a,2}$};
\node(3) at (0,-1){$\set{a,3}$};
\node(4) at (0.95,-0.309){$\set{a,4}$};

\path[draw,bend left,decoration={
    markings,
    mark=at position 0.5 with {\arrow[black]{stealth};}}]
(y) edge[postaction=decorate] (0)
(0) edge[postaction=decorate] (y)
(y) edge[postaction=decorate] (1)
(1) edge[postaction=decorate] (y)
(y) edge[postaction=decorate] (2)
(2) edge[postaction=decorate] (y)
(y) edge[postaction=decorate] (3)
(3) edge[postaction=decorate] (y)
(y) edge[postaction=decorate] (4)
(4) edge[postaction=decorate] (y)
;
\end{tikzpicture}
\end{figure}
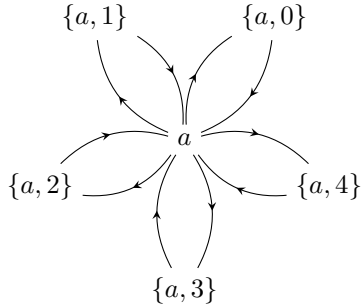

\begin{lemma}\label{lemma:bouquetsexist}
  Let $A_0$ be a finite subset of $\omega\setminus\set{0}$ and let $M\models \mathsf{ZFA}$. Then $M_1$ contains an  $A_0$-bouquet.
\end{lemma}
\begin{proof}
It suffices to find a certain finite graph as a connected component of $M_1$, so this follows  from Proposition~\ref{pr:embedgraphs} (or directly from~\cite[Theorem~4]{adamdaycameron}).
\end{proof}
If $M$ is a structure, denote by $\Th(M)$ its theory.

\begin{pr}\label{pr:contrtypes}
  Let $M\models\mathsf{ZFA}$.  Then in $\Th(M_1)$  the $2^{\aleph_0}$ sets of formulas $\beta_A$, for $A\subseteq \omega\setminus \set0$, are each consistent, and pairwise contradictory.  In particular, the same is true in $\Th(M)$.
\end{pr}
\begin{proof}
  If $A, B$ are distinct subsets of $\omega\setminus\set 0$ and, without loss of generality, there is an $n\in A\setminus B$, then $\beta_A$ contradicts $\beta_B$ because $\beta_A(y)\proves \exists x_n\; (\phi_n(x_n)\land D(y,x_n))$ and  $\beta_B(y)\proves \neg \exists x_n\; (\phi_n(x_n)\land D(y,x_n))$.
  
  To show that each $\beta_A$ is consistent it is enough, by compactness, to show that if $A_0$ is a finite subset of $A$ and $A_1$ is a finite subset of  $\omega\setminus(\set{0}\cup A)$ then there is some $b\in M$ with a $D$-edge to an $n$-flower for every $n\in A_0$ and no $D$-edges to $n$-flowers whenever $n\in A_1$. Any $A_0$-bouquet will satisfy these requirements and, by Lemma~\ref{lemma:bouquetsexist}, an $A_0$-bouquet exists inside $M_1$.
  
  For the last part, note that all the theories at hand are complete (in different languages), and whether or not an intersection of definable sets is empty does not change after adding more definable sets.
\end{proof}
To conclude, we need the following standard fact from model theory.
\begin{fact}\label{fact:manytypes}
Every partial type over $\emptyset$ of a countable theory can be realised in a countable model.
\end{fact}
\begin{co}\label{co:manymodels}
  Let $M$ be a model of $\mathsf{ZFA}$. There are $2^{\aleph_0}$ countable models of $\mathsf{ZFA}$ such that their D-graphs (resp.~SD-graphs) are elementarily equivalent to $M_1$ (resp.~$M_0$) and  pairwise non-isomorphic. 
\end{co}
\begin{proof}
  Consider the pairwise contradictory partial types $\beta_A$. By Fact~\ref{fact:manytypes}, $\Th(M)$ has $2^{\aleph_0}$ distinct countable models, as each of them can only realise countably many of the $\beta_A$. The reducts to $L_1$ (resp.~$L_0$) of models realising different subsets of  $\set{\beta_A\mid A\subseteq \omega\setminus\set 0}$ are still non-isomorphic, since the $\beta_A$ are partial types in the language $L_1$.
\end{proof}
 The previous Corollary answers affirmatively~\cite[Questions~1 and~2]{adamdaycameron}.

\begin{rem}
For the results in this section to hold, it is not necessary that $M$ satisfies the whole of $\mathsf{ZFA}$. It is enough to be able to prove Lemma~\ref{lemma:bouquetsexist} for $M$, and it is easy to see than one can provide a direct proof whenever in $M$ it is possible to define infinitely many different well-founded sets, e.g.\ von Neumann natural numbers, and to ensure existence of solutions to flat systems of equations. This can be done as long as $M$ satisfies Extensionality, Empty Set, Pairing,  and $\mathsf{AFA}_1$\footnote{Stated using a sensible coding of flat systems, which can be carried out using Pairing.}.  If we replace, in Definition~\ref{defin:flatsystem}, `$x=S_x$' with `$x$ and $S_x$ have the same elements', then we can even drop Extensionality.
\end{rem}

\section{Common Theory}\label{sec:theory}
The main aim of this section is to study the common theory   of the class of D-graphs of $\mathsf{ZFA}$. We show in Corollary~\ref{co:concon} that it is incomplete,  and in Corollary~\ref{co:completions} characterise its completions  in terms of collections of consistency statements. Furthermore, we show that each of these completions is untame in the sense of neostability theory (Corollary~\ref{co:untame}) and has a countable model which is not a D-graph, and that the same holds for SD-graphs (Corollary~\ref{co:q5}), therefore solving negatively~\cite[Question~5]{adamdaycameron}.

\begin{defin}
  Let $K_1$ be the class of D-graphs of models of $\mathsf{ZFA}$. Let $\Th(K_1)$ be its common $L_1$-theory.
\end{defin}

\begin{defin}
  Let $\phi$ be an $L_1$-sentence. We define an $L_1$-sentence $\mu(\phi)$ as follows.  Let $x$ be a variable not appearing in $\phi$. Let $\chi(x)$ be obtained from $\phi$ by relativising $\exists y$ and $\forall y$ to $D(x,y)$. Let $\mu(\phi)$ be the formula $\exists x\; (\neg D(x,x)\land \chi(x))$. 
\end{defin}
In other words, $\mu(\phi)$ can be thought of as saying that there is a point whose set of neighbours is a model of $\phi$.

\begin{rem}
Suppose $\phi$ is a `standard' sentence, i.e.\ one which is a formula in the sense of the metatheory, say in the finite language $L'$. Let $M\models \mathsf{ZFA}$, and let $N$ be an $L'$-structure in $M$. Then, whether $N\models \phi$ or not is absolute between $M$ and the metatheory. Every formula we mention is of this kind, and this fact will be used tacitly from now on.
\end{rem}
\begin{defin}
  Let $\Phi$ be the set of $L_1$-sentences which imply $\forall x,y\;(D(x,y)\implica D(y,x))$.
\end{defin}
\begin{lemma}\label{lemma:muworks}
  For every $L_1$-sentence $\phi\in \Phi$ and every $M\models \mathsf{ZFA}$ we have
    \[
M\models \Con(\phi)\iff M_1\models \mu(\phi)
\]
Moreover, if this is the case, then there is $H\subseteq M_1$ such that
  \begin{enumerate}
  \item $(H, D^{M_1}\restr H)$ satisfies $\phi$,
  \item $H$ is a union of regions of $M$, and
  \item $H$ is an $M$-set.
  \end{enumerate}
  \end{lemma}
  \begin{proof}
    Note that the class of graphs in $M$ is closed under the operations of removing a point or adding one and connecting it to everything. Now apply Proposition~\ref{pr:embedgraphs}.
  \end{proof}
Define $L_{\mathsf{NBG}}\coloneqq\set{E}$, where $E$ is a binary  relational symbol.   We think of $L_1$ as `the language of graphs' and  of $L_{\mathsf{NBG}}$ as `the language of digraphs', specifically, digraphs that are models of a certain class theory (see below), hence the notation. It is well-known that every digraph is interpretable in a graph, and that such an interpretation may be chosen to be uniform, in the sense below. See e.g.~\cite[Theorem~5.5.1]{hodges}.

\begin{fact}\label{fact:interpretation}
  Every $L_{\mathsf{NBG}}$-structure $N$ is interpretable in a graph $N'$. Moreover, for every $L_{\mathsf{NBG}}$-sentence $\theta$ there is an $L_1$-sentence $\theta'$ such that
  \begin{enumerate}
  \item $\theta$ is consistent if and only if $\theta'$ is, and
  \item for every $L_{\mathsf{NBG}}$-structure $N$ we have $N\models \theta\iff N'\models\theta'$.
  \end{enumerate}
\end{fact}
  \begin{co}\label{co:muphiprime}
  For every  $L_{\mathsf{NBG}}$-sentence $\theta$, let $\theta'$ be as in Fact~\ref{fact:interpretation}. For all $M\models \mathsf{ZFA}$ \[M\models \Con(\theta)\iff M_1\models \mu(\theta')\]
\end{co}
\begin{proof}
Apply Lemma~\ref{lemma:muworks} to $\phi\coloneqq \theta'$.
\end{proof}

\begin{co}\label{co:untame}
Let $M\models \mathsf{ZFA}$. Then every model of $\Th(M_1)$ interprets  with parameters  arbitrarily large finite fragments of $\mathsf{ZFC}$. In particular $\Th(M_1)$ has  $\mathsf{SOP}$, $\mathsf{TP_2}$, and $\mathsf{IP}_k$ for all $k$. 
\end{co}
\begin{proof}
If $\theta$ is the conjunction of a finite fragment of $\mathsf{ZFC}$, it is well-known that $\mathsf{ZFA}\proves \Con(\theta)$.  Since a model of  $\theta$ is a digraph,  we can apply Corollary~\ref{co:muphiprime}. If $a$ witnesses the outermost existential quantifier in $\mu(\theta')$, then $\theta$ is interpretable with parameter $a$.
\end{proof}
We now want to use Corollary~\ref{co:muphiprime} to show that the common theory $\Th(K_1)$ of the class of D-graphs of models of $\mathsf{ZFA}$ is incomplete. Naively, this could be done by choosing $\theta$ to be a finite axiomatisation of some theory equiconsistent with $\mathsf{ZFA}$, and then invoking the Second Incompleteness Theorem. For instance, one could choose von Neumann-Bernays-G\"odel class theory $\mathsf{NBG}$, axiomatised in the language $L_\mathsf{NBG}$\footnote{The reader may have encountered an axiomatisation using two sorts; this can be avoided by declaring sets to be those classes that are elements of some other class.}, as this is known to be equiconsistent with $\mathsf{ZFC}$ (see~\cite{felgner}), hence with $\mathsf{ZFA}$. The problem with this argument is that, in order for it to work, we need a further set-theoretical assumption in our metatheory, namely $\Con(\mathsf{ZFC}+\Con(\mathsf{ZFC}))$. This can be avoided by using another sentence whose consistency is independent of $\mathsf{ZFA}$, provably in $\mathsf{ZFC}+\Con(\mathsf{ZFC})$ alone. We would like to thank Michael Rathjen for pointing out to us the existence of such a sentence.

Let $\mathsf{NBG}^-$ denote $\mathsf{NBG}$ without the axiom of Infinity. We will use  special cases of a classical theorem of Rosser and of a related result. For proofs of these, together with their more general statements,  we refer the reader to~\cite[Chapter~7, Application~2.1 and Corollary~2.6]{smorynskitw}.

\begin{fact}[Rosser's Theorem]
  There is a $\Pi^0_1$ arithmetical statement $\psi$ which is independent of $\mathsf{ZFA}$.
\end{fact}
\begin{fact}\label{fact:eqcon}
  Let $\psi$ be a $\Pi^0_1$ arithmetical statement.  There is another arithmetical statement $\widetilde \psi$ such that $\mathsf{ZFA}\proves \psi\coimplica \Con(\mathsf{NBG}^-+\widetilde\psi)$.
\end{fact}

\begin{co}\label{co:concon}
$\Th(K_1)$ is not complete. 
\end{co}
\begin{proof}
Let $\psi$ be given by Rosser's Theorem, and let $\widetilde\psi$ be given by Fact~\ref{fact:eqcon} applied to $\psi$. Apply  Corollary~\ref{co:muphiprime} to $\theta\coloneqq\mathsf{NBG}^-+\widetilde\psi$.
\end{proof}

It is therefore natural to study the completions of $\Th(K_1)$, and it follows easily from $K_1$ being pseudoelementary that all of these are the theory of some actual D-graph $M_1$. We provide a proof for completeness.

\begin{pr}\label{pr:completions}
  Let $T$ be an $L$-theory, and let  $K$ be the class of its models. Let $L_1\subseteq L$, and for $M\in K$ denote $M_1\coloneqq M\restr L_1$. Let $K_1\coloneqq\set{M_1\mid M\in K}$ and $N\models \Th(K_1)$. Then there is $M\in K$ such that $M_1 \equiv N$. 
\end{pr}
\begin{proof}
We are asking whether there is any $M\models T\cup \Th(N)$, so it is enough to show that the latter theory is consistent. If not, there is an $L_1$-formula $\phi\in \Th(N)$ such that $T\proves \neg \phi$. In particular, since $\neg\phi\in L_1$, we have that $\Th(K_1)\proves \neg \phi$, and this contradicts that $N\models \Th(K_1)$.
\end{proof}

In order to  characterise the completions of $\Th(K_1)$, we will use techniques from finite model theory, namely Ehrenfeucht-Fra\"iss\'e games and $k$-equivalence. For background on these concepts, see~\cite{ebbinghausflum}.

\begin{lemma}\label{lemma:duog}
  Let $G=G_0\sqcup G_1$ be a graph with no edges between $G_0$ and $G_1$, and let $H=H_0\sqcup H_1$ be a graph with no edges between $H_0$ and $H_1$. If $(G_0, \bla a1,{m-1})\equiv_k (H_0,\bla b1,{m-1})$ and $(G_1, a_m)\equiv_k (H_1, b_m)$, then $(G,\bla a1,m)\equiv_k (H, \bla b1,m)$.
\end{lemma}
\begin{proof}
This is standard, see e.g.~\cite[Proposition~2.3.10]{ebbinghausflum}.
\end{proof}

\begin{thm}\label{thm:completions}
  Let $M$ and $N$ be models of $\mathsf{ZFA}$.   The following are equivalent.
  \begin{enumerate}
  \item\label{point:m1n1} $M_1\equiv N_1$.
  \item\label{point:Phi} $M_1$ and $N_1$ satisfy the same sentences of the form $\mu(\phi)$, as $\phi$ ranges in $\Phi$.
  \item \label{point:consistency}$M$ and $N$ satisfy the same consistency statements.
  \end{enumerate}
\end{thm}
\begin{proof}
For statements about graphs, the equivalence of~\ref{point:Phi} and~\ref{point:consistency} follows from Lemma~\ref{lemma:muworks}. For statements in other languages, it is enough to interpret them in graphs using~\cite[Theorem~5.5.1]{hodges}.
  
For the equivalence of~\ref{point:m1n1} and~\ref{point:Phi}, we show that for every $n\in \omega$ the Ehrenfeucht-Fra\"iss\'e game between $M_1$ and $N_1$ of length  $n$ is won by the Duplicator, by describing a winning strategy. The idea behind the strategy is the following. Recall that, for every finite relational language and every $k$, there is only a finite number of $\equiv_{k}$-classes, each characterised by a single sentence (see e.g.~\cite[Corollary~2.2.9]{ebbinghausflum}).  After the Spoiler plays a point $a$, the Duplicator replicates  the $\equiv_k$-class of the region of $a$  using Lemma~\ref{lemma:muworks}.

Fix the length $n$ of the game and denote by $\bla a1,{m}\in M_1$ and $\bla b1,{m}\in N_1$ the points chosen  at the end of turn $m$.
 The Duplicator defines, by simultaneous induction on $m$,  sets $G^{m}_0\subseteq M_1$ and $H^{m}_0\subseteq N_1$, and  makes sure  that they satisfy the following conditions.
\begin{enumerate}[label=(C\arabic*)]
\item \label{point:trivial} $\bla a1,m\in G^m_0$ and $\bla b1,m \in H^m_0$.
\item \label{point:selfcontained} $G^m_0$ and $H^m_0$ are unions of regions of $M$ and $N$ respectively.
  \item \label{point:aresets} $G_0^m$ and $H_0^m$ are respectively an $M$-set and an $N$-set.
\item \label{point:stategycondition} When $G^m_0$ and $H^m_0$ are equipped with the $L_1$-structures induced by $M$ and $N$ respectively, we have $(G^m_0,\bla a1,m)\equiv_{n-m} (H^m_0, \bla b1,m)$.
\end{enumerate}
Before the game starts (`after turn $0$') we set $G^0_0=H^0_0=\emptyset$ and all conditions trivially hold. Assume inductively that they hold after turn $m-1$. We deal with the case where the Spoiler plays  $a_m\in M_1$; the case where the Spoiler plays $b_m \in N_1$ is symmetrical. 

Let $G^{m}_1$ be the  region of $a_m$ in $M$. If $G_1^m\subseteq G_0^{m-1}$ then, since by inductive hypothesis condition~\ref{point:stategycondition} held after turn $m-1$, the Duplicator can find $b_m\in H_0^{m-1}$  such that $(G_0^{m-1},\bla a0,m)\equiv_{n-m}(H_0^{m-1},\bla b0,m)$. It is then clear that all conditions hold after setting $G_0^m=G_0^{m-1}$ and $H_0^m=H_0^{m-1}$. 

Otherwise, by~\ref{point:selfcontained}, we have  $G^m_1\cap G^{m-1}_0=\emptyset$.    Let $\phi$  characterise the $\equiv_{n-m+1}$-class of $G^m_1$. Note that, if $n-m+1\ge 2$, then $\phi\in \Phi$ automatically. Otherwise, replace $\phi$ with $\phi\land \forall x\forall y\; (D(x,y)\implica D(y,x))$. By Remark~\ref{rem:compdef}, $G^m_1$ is an $M$-set, hence $M\models \Con(\phi)$. By Lemma~\ref{lemma:muworks} and assumption, there is a union $H^m_1$ of regions of $N$ which is an $N$-set and such that $G^m_1\equiv_{n-m+1} H^m_1$.  By inductive hypothesis,  $H_0^{m-1}$ is also an $N$-set by~\ref{point:aresets}. Therefore, up to writing a suitable flat system in $N$, we may  replace $H_1^m$ with an isomorphic copy which is still a union of regions and an $N$-set, but with $H_1^m\cap H_0^{m-1}=\emptyset$.

Let $b_m\in H_1^m$ be the choice given by a winning strategy for the Duplicator in the game of length $n-m+1$ between $G_1^m$ and $H_1^m$ after the Spoiler plays $a_m\in G_1^m$ as its first move. Set $G^m_0=G^{m-1}_0\cup G^m_1$ and $H^m_0=H^{m-1}_0\cup H^m_1$.  Note that $G^{m-1}_0, G^m_1, H^{m-1}_0, H^m_1$ are all unions of regions and $M$-sets or $N$-sets, hence~\ref{point:selfcontained} and~\ref{point:aresets} hold (and~\ref{point:trivial}  is clear). Moreover both unions are disjoint, so the hypotheses of  Lemma~\ref{lemma:duog} are satisfied and $(G^m_0,\bla a1,m) \equiv_{n-m} (H^m_0, \bla b1,m)$, i.e.~\ref{point:stategycondition} holds.

To show that this strategy is winning, note that the outcome of the game only depends on the induced structures on $\bla a1,n$ and $\bla b1,n$ at the end of the final turn.  These do not depend on what is outside $G_0^n$ and $H_0^n$ since they are unions of regions, hence unions of connected components. As~\ref{point:stategycondition} holds at the end of turn $n$, the structures induced on $\bla a1,n$ and $\bla b1,n$ are isomorphic.
\end{proof}
\begin{co}\label{co:completions}
  Let $N\models \Th(K_1)$. Then $\Th(N)$ is axiomatised by
  \[
    \Th(K_1)\cup \set{\mu(\phi)\mid \phi\in \Phi, N\models \mu(\phi)}\cup \set{\neg \mu(\phi)\mid \phi\in \Phi, N\models\neg\mu(\phi)}
  \]
\end{co}
\begin{proof}
Let $N'$ satisfy the axiomatisation above. Since $N$ and $N'$ are  models of $\Th(K_1)$ we may, by Proposition~\ref{pr:completions}, replace  them with D-graphs $M_1\equiv N$ and $M_1'\equiv N'$ of models of $\mathsf{ZFA}$. By Theorem~\ref{thm:completions} $M_1\equiv M_1'$.
\end{proof}
By the previous corollary, combined with Lemma~\ref{lemma:muworks},  theories  of double-membership graphs correspond bijectively to consistent  (with $\mathsf{ZFA}$, equivalently with $\mathsf{ZFC}$) collections of consistency statements.

The reader familiar with finite model theory may have noticed similarities between the proof of  Theorem~\ref{thm:completions} and certain proofs of the theorems of Hanf and Gaifman (see~\cite[Theorems~2.4.1 and~2.5.1]{ebbinghausflum}). In fact one could deduce a statement similar to Theorem~\ref{thm:completions} directly from  Gaifman's Theorem. This would characterise the completions of $\Th(K_1)$ in terms of \emph{local formulas}, of which the $\mu(\phi)$ form a subclass, yielding a   less specific result than  Corollary~\ref{co:completions}. Moreover, we believe that the correspondence with collections of consistency statements provides a conceptually clearer picture.

Similar ideas can be used to study~\cite[Question~5]{adamdaycameron}, which asks whether  a countable structure elementarily equivalent to the SD-graph $M_0$ of some $M\models \mathsf{ZFA}$ must itself be the SD-graph of some model of $\mathsf{ZFA}$. We provide a negative solution in Corollary~\ref{co:q5}. Again, Gaifman's Theorem could be used directly to deduce its second part.

\begin{thm}\label{thm:noinfdiam}
  Let $M\models \mathsf{ZFA}$. There is a countable $N\equiv M_0$ such that $N\restr L_1$ has no connected component of infinite diameter.
\end{thm}
Before the proof, we show how this solves~\cite[Question~5]{adamdaycameron}. 

\begin{co}\label{co:q5}
  For every $M\models \mathsf{ZFA}$ there are a countable  $N\equiv M_0$ which is not the SD-graph of any model of $\mathsf{ZFA}$ and a countable  $N'\equiv M_1$ which is not the D-graph of any model of $\mathsf{ZFA}$.   
\end{co}
\begin{proof}
  Let $N$ be given by Theorem~\ref{thm:noinfdiam} and $N'\coloneqq N\restr L_1$. Now observe that, as follows easily from Proposition~\ref{pr:embedgraphs}, any reduct to $L_1$ of a model of $\mathsf{ZFA}$ has a connected component of infinite diameter. 
\end{proof}
Note that this proves slightly more: a negative solution to the question would only have required to find a single pair $(M_0, N)$ satisfying the conclusion of the corollary.

\begin{proof}[Proof of Theorem~\ref{thm:noinfdiam}]
  Up to passing to a countable elementary substructure, we may assume that $M$ itself is countable. Let $N$ be obtained from $M_0$ by removing all points whose connected component in $M_1$ has infinite diameter. We show that  $M_0\equiv N$ by exhibiting, for every $n$, a sequence $(I_j)_{j\le n}$ of non-empty sets of partial isomorphisms between $M_0$ and $N$ with the back-and-forth property (see~\cite[Definition~2.3.1 and Corollary~2.3.4]{ebbinghausflum}).  The idea  is to adapt the proof of~\cite[Lemma~2.2.7]{otto} (essentially Hanf's Theorem) by considering the Gaifman balls with respect to $L_1$, while requiring the partial isomorphisms  to preserve the richer language $L_0$.
  
 On an $L_0$-structure $A$, consider the distance $d\from A\to \omega\cup \set \infty$ given by the graph distance in the reduct $A\restr L_1$ (where $d(a,b)=\infty$ iff  $a,b$ lie in distinct connected components). If $\bla a1,k\in A$ and $r\in \omega$, denote by $\dom(B(r,\bla a1,k))$ the union of the balls of radius $r$ (with respect to $d$) centred on $\bla a1,k$. Equip $\dom(B(r,\bla a1,k))$ with the $L_0$-structure induced by $A$, then expand  to an  $L_0\cup\set{\bla c1,k}$-structure $B(r, \bla a1,k)$ by interpreting each constant symbol $c_i$ with the corresponding $a_i$. We stress that, even though  $B(r, \bla a1,k)$ carries an $L_0\cup\set{\bla c1,k}$-structure, and  we consider  isomorphisms with respect to this structure, the balls giving its domain are defined with respect to the distance induced by $L_1$ alone.

Set $r_j\coloneqq (3^{j}-1)/2$ and fix $n$.  Define $I_n\coloneqq\set{\emptyset}$, where $\emptyset$ is thought of as the empty partial map $M_0\to N$.  For $j<n$, let $I_j$  be the following  set of partial maps    $M_0\to N$:
  \[
I_j\coloneqq\set{\bla a1,k\mapsto \bla b1,k\mid k\le n-j, B(r_j, \bla a1,k)\cong B(r_j, \bla b1,k)}
  \]
We have to show that for every map $\bla a1,k\mapsto \bla b1,k$ in $I_{j+1}$ and every $a\in M_0$ [resp.~every $b\in N$] there is $b\in N$ [resp.~$a\in M_0$] such that $\bla a1,k, a\mapsto \bla b1,k, b$ is in  $I_j$.

Denote by $\iota$ an isomorphism $B(r_{j+1}, \bla a1,k)\to B(r_{j+1}, \bla b1,k)$ and let $a\in M_0$. If $a$ is chosen in $B(2\cdot r_j+1, \bla a1,k)$, then by the triangle inequality and the fact that $2\cdot r_j+1+r_j=r_{j+1}$ we have  $B(r_j, a)\subseteq B(r_{j+1}, \bla a1,k)$, and we can just set $b\coloneqq \iota(a)$.

Otherwise, again by the triangle inequality,  $B(r_j, a)$ and $B(r_{j}, \bla a1,k)$ are disjoint and there is no $D$-edge between them. Note, moreover, that they are $M$-sets. This  allows us to write a suitable flat system, which will yield the desired $b$.

Working inside $M$, for every $d\in B(r_j, a)$ choose a well-founded set $h_d$ such that for all $d,d_0,d_1\in B(r_j,a)$ we have 
\begin{enumerate}[label=(H\arabic*)]
\item \label{point:noedge} $h_{d_0}\notin h_{d_1}$,
\item \label{point:distinct}if $d_0\ne d_1$ then $h_{d_0}\ne h_{d_1}$,
  \item \label{point:nocup}$h_d\notin  B(r_{j}, \bla b1,k)$, 
\item \label{point:cup}$h_d\notin \bigcup B(r_{j}, \bla b1,k)$, and
\item \label{point:cupcup}$h_d\notin \bigcup \bigcup B(r_{j}, \bla b1,k)$.
\end{enumerate}
Let $\set{x_d\mid d\in B(r_j,a)}$ be a set of indeterminates. Define
\begin{align*}
  P_d&\coloneqq\set{x_{e}\mid e\in B(r_j, a), M\models e\in d} \\
  Q_d&\coloneqq\set{\iota(f)\mid f\in B(r_j, \bla a1,k), M\models S(d,f)}
\end{align*}
and consider the flat system
\begin{equation}
  \set{x_d=\set{h_d}\cup P_d \cup Q_d\mid d\in B(r_j, a)}\label{eq:lastflatsystem}\tag{$*$}
\end{equation}
Intuitively, the terms $P_d$ ensure that the image of a solution is an isomorphic copy of $B(r_j, a)$, while the terms $Q_d$ create the appropriate $S$-edges between the image and $B(r_j,\bla b1,k)$ (note that we do not need any $D$-edges because there are none between $B(r_j,a)$ and $B(r_j, \bla a1,k)$). The $\set{h_d}$ are needed for bookkeeping reasons, in order to avoid pathologies. We now spell out the details; keep in mind that each $P_d$ consists of indeterminates, and each $Q_d$ is a subset of $B(r_j, \bla b1,k)$.

Let $s$ be a solution of~\eqref{eq:lastflatsystem}, guaranteed to exist by $\mathsf{AFA}$.  By~\ref{point:noedge} and the fact that each member of  $\operatorname{Im}(s)$ contains some $h_d$, we have $\set{h_d\mid d\in B(r_j,a)}\cap \operatorname{Im}(s)=\emptyset$.  Using this together with~\ref{point:distinct} and~\ref{point:nocup} we have $h_d\in s(x_e)\iff d=e$, hence $s$ is injective.

Let $s'\coloneqq d\mapsto s(x_d)$  and $b\coloneqq s'(a)$. By~\ref{point:cup} we have that   $\operatorname{Im}(s)$ does not intersect $B(r_j, \bla b1,k)$, and we already showed that it does not meet $\set{h_d\mid d\in B(r_j,a)}$.   By looking at~\eqref{eq:lastflatsystem} and at the  definition of the terms $P_d$, we have that  $\operatorname{Im}(s)=B(r_j, b)$ and that $s'$ is an isomorphism $B(r_j,a)\to B(r_j, b)$.

Note that the only $D$-edges involving points of $\operatorname{Im}(s)$ can come from the terms $P_d$: the $h_d$ are well-founded, and   there are no $g\in \operatorname{Im}(s)$ and $\ell\in B(r_{j}, \bla b1,k)$ such that $g\in \ell$, since $g$ contains some $h_d$ but this cannot be the case for any element of $\ell$  because of~\ref{point:cupcup}. Hence $\operatorname{Im}(s)$ is a connected component of $M_1$ and it has  diameter not exceeding $2\cdot r_j$, so is included in $N$. 

Set $\iota'\coloneqq s'\cup(\iota \restr B(r_j,\bla a1,k))$. This map is injective because it is the union of two injective maps whose images $B(r_j, b)$  and $B(r_j, \bla b1,k)$ are, as shown above, disjoint. Moreover, there are no $D$-edges between $B(r_j, b)$ and $B(r_j, \bla b1,k)$,  since the former is a connected component of $M_1$. By inspecting the terms $Q_d$, we conclude that $\iota'$ is an isomorphism  $B(r_j, \bla a1,k,a)\to B(r_j, \bla b1,k,b)$, and this settles the `forth' case.

The proof of the `back' case, where we are given $b\in N$ and need to find $a\in M_0$, is analogous (and shorter, as we do not need to ensure that the new points are in $N$): we can  consider statements such as $e\in d$ when $e,d\in N$ since the domain of the $L_0$-structure $N$ is a subset of $M$.
\end{proof}
\begin{paragraph}{Problems}
We leave the reader with some open problems.
  \begin{enumerate}
  \item   Axiomatise the theory of D-graphs of models of $\mathsf{ZFA}$.
  \item   Axiomatise the theory of SD-graphs of models of $\mathsf{ZFA}$.
  \item Characterise the completions of the theory of SD-graphs of models of $\mathsf{ZFA}$.
  \end{enumerate}
\end{paragraph}

\begin{paragraph}{Acknowledgements}
We are grateful to Michael Rathjen for pointing out to us Fact~\ref{fact:eqcon}, and to Dugald Macpherson and Vincenzo Mantova for their guidance and feedback. The first author is supported by a Leeds Doctoral Scholarship. The second and third authors are supported by Leeds Anniversary Research Scholarships.
\end{paragraph}

\end{document}